\documentclass[12pt]{article}
\usepackage[utf8]{inputenc}
\usepackage{amsmath}
\usepackage{amssymb}
\usepackage{mathtools}
\usepackage{amsthm}
\usepackage{hyperref}
\usepackage{bm, upgreek}
\usepackage{graphicx}
\usepackage{caption}
\usepackage{subcaption}
\usepackage{natbib}
\usepackage{todonotes}
\usepackage{tikz}
\usepackage{pgfplots}
\usepackage{floatrow}
\usetikzlibrary{shapes}
\usetikzlibrary{plotmarks}

\usetikzlibrary{positioning,calc}
\usetikzlibrary{backgrounds}
    
\pgfplotsset{compat=1.16}
\definecolor{bgColor2}{HTML}{F5793A}
\definecolor{dataColor}{HTML}{F5793A}
\definecolor{axisColor}{RGB}{0,0,0}
\definecolor{labelColor}{gray}{0.30}
\definecolor{tickColor}{gray}{0.30}
\definecolor{textColor}{RGB}{0,0,0}
\definecolor{frameColor}{RGB}{255,255,255}
\definecolor{lineColor}{HTML}{0F2080}
\definecolor{fillColor}{RGB}{200,200,200}
\definecolor{outlineColor}{gray}{0.30}
\definecolor{refColor}{RGB}{216,27,96}
\definecolor{lcolor1}{RGB}{216,27,96}
\definecolor{lcolor2}{RGB}{30,136,229}
\definecolor{lcolor3}{RGB}{255,193,7}
\definecolor{lcolor4}{RGB}{0,77,64}

\title{Sequential CI}
\author{Mattias Nordin and Mårten Schultzberg}

\newtheorem*{proof*}{Proof}

\newtheorem{lemma}{Lemma}
\newtheorem{corollary}{Corollary}

\theoremstyle{plain}
\newtheorem{theorem}{Theorem}

\theoremstyle{definition}
\newtheorem{definition}{Definition}

\theoremstyle{remark}
\newtheorem{remark}{Remark}

\newcommand{\var}{\mathbb{V}}
\renewcommand{\Pr}{\mathbb{P}}
\newcommand{\E}{\mathbb{E}}

\usepackage{sistyle}
\SIthousandsep{,}

\newcommand{\ind}{\perp\!\!\!\!\perp}

\makeatletter
\newcommand*{\indep}{%
  \mathbin{%
    \mathpalette{\@indep}{}%
  }%
}
\newcommand*{\nindep}{%
  \mathbin{
    \mathpalette{\@indep}{\not}
  }%
}
\newcommand*{\@indep}[2]{%
  \sbox0{$#1\perp\m@th$}
  \sbox2{$#1=$}
  \sbox4{$#1\vcenter{}$}
  \rlap{\copy0}
  \dimen@=\dimexpr\ht2-\ht4-.2pt\relax
  \kern\dimen@
  {#2}%
  \kern\dimen@
  \copy0 
} 
\makeatother

\newcommand{\as}{\text{ a.s.}}

\begin{document}
\title{Precision-based designs for sequential randomized experiments}
\author{Mattias Nordin\footnote{Department of Statistics, Uppsala University. Email: mattias.nordin@statistics.uu.se} and M\aa rten Schultzberg\footnote{Experimentation Platform team, Spotify. Email: mschultzberg@spotify.com}}
\maketitle
\begin{abstract}
    In this paper, we consider an experimental setting where units enter the experiment sequentially. Our goal is to form stopping rules which lead to estimators of treatment effects with a given precision. We propose a fixed-width confidence interval design (FWCID) where the experiment terminates once a pre-specified confidence interval width is achieved. We show that under this design, the difference-in-means estimator is a consistent estimator of the average treatment effect and standard confidence intervals have asymptotic guarantees of coverage and efficiency for several versions of the design. In addition, we propose a version of the design that we call fixed power design (FPD) where a given power is asymptotically guaranteed for a given treatment effect, without the need to specify the variances of the outcomes under treatment or control. In addition, this design also gives a consistent difference-in-means estimator with correct coverage of the corresponding 
 standard confidence interval. We complement our theoretical findings with Monte Carlo simulations where we compare our proposed designs with standard designs in the sequential experiments literature, showing that our designs outperform these designs in several important aspects. We believe our results to be relevant for many experimental settings where units enter sequentially, such as in clinical trials, as well as in online A/B tests used by the tech and e-commerce industry.
\end{abstract}

\section{Introduction}
A key component in the traditional design of randomized experiments is to decide on the sample size. The experimenter may consider many different aspects such as costs, time, ethics and urgency when deciding on the size of the experiment, but it is generally agreed upon that a well-designed experiment should include a thorough power analysis. In an underpowered study, statistically significant effects are generally overestimates of the true effect size \citep{gelman2014}, whereas in a overpowered study, an unnecessary amount of resources are spent in the experiment.

In many experimental settings, units enter into the experiments sequentially. This is, for instance, the case in many clinical trials but also in online A/B testing in the tech and e-commerce industry. In such cases, instead of relying on a fixed-sample design (where the number of observations are decided in advance) the experimenter can ensure that resources are not wasted by continuously monitoring the results and terminate the experiment as soon as results are sufficiently clear to allow a conclusion to be drawn.
To do so in a statistically rigorous way, a decision rule which take into account that results are monitored continuously need to be put in place.
Sequential tests, originating with \cite{wald1945}, have been developed to do just that. For instance, with the error-spending approach of \cite{lan_discrete_1983}, the experimenter decides on how much type I error to spend at any given point such that the overall type I error is set at a given value, for instance at five percent. If, at any given point, the observed statistic exceeds the implied threshold value, the experiment is terminated. Another sequential testing framework that has gained a lot of attention, especially in the online experimentation literature, is the so-called ``always-valid inference'' \citep{johari2015always, johari2022always, howard2021, waudby2021time}, which bounds the false positive rate by any desired $\alpha$ under continuous monitoring, where the sample size is allowed to approach infinity.

The sequential testing approach is useful when the relevant question is of the form ``is the treatment better than the control''. However, for many research questions, it is also of interest how large the difference is. In such cases sequential tests will in general work poorly as the usual difference-in-means estimator is a biased estimator of the average treatment effect when early stopping is allowed \citep{fan2004conditional}. Furthermore, the cost of using sequential tests is that---compared to fixed-sample designs---the power of the test will be lower, because it is necessary to adjust for multiple testing. The lower power is especially acute for always-valid tests who tend to be very conservative. Finally, if the null hypothesis is true, sequential designs will have to run until a final, pre-specified, maximum sample size (or time) is reached, resulting in potentially inefficient use of data (although, sequential designs may also include futility bounds to mitigate this problem)

In this paper we propose an alternative stopping rule to significance-based stopping, where the goal of the design is to achieve a confidence interval of a given pre-specified  width. This width is achieved by continuously forming confidence intervals for the treatment effect and stopping once a targeted width is reached. We call this design a ``fixed-width confidence interval design'' (FWCID) and introduce both a naive and conservative version of the design. We show that both versions results in consistent estimators of the average treatment effect with asymptotically correct coverage of the confidence interval. In addition, the use of data is asymptotically efficient in the sense that the expected sample size is the same as the sample size that would have been chosen in a fixed-sample design to achieve the target confidence interval, had the variances of the outcome under treatment and control been known. Moreover, the design makes it straightforward to---at any time during the experiment---make a forecast of at what sample size the experiment will terminate, something that is useful from a practical perspective.

Finally, we also show how the idea of forming a stopping rule based on the confidence interval width, i.e., a function of the variance of the treatment effect estimator, can be translated into a design for hypothesis testing, which we call a fixed-power design (FPD). The perhaps most difficult aspect of a fixed-sample design when performing a power analysis is to form beliefs about the outcome distribution under both treatment and control. Even though historical outcome data is sometimes available, it is not trivial to estimate the variance of the treatment effect estimator. For instance, in online experiments, the outcome distribution often changes over time due to seasonality and growing and/or changing user bases, and the fact that treatment effects are often heterogeneous. With heterogeneous treatment effects, the variance under no treatment may not be a good proxy for the variance under treatment. With the FPD, there is no need to guess what the outcome distribution will look like in the design stage. Instead, power is assured by letting the final sample size be stochastic. This feature could be undesirable if the experimenter do not want the experiment to run for ``too long''. However, we also show that the FPD admits continuous monitoring of expected sample size at which the experiment will terminate. Therefore, the experimenter can already from an early stage check when the experiment will terminate and take appropriate action (such as terminating the experiment) if the expected sample size is too high. In addition, the fact that sample size is stochastic is a guard against overly optimistic power analyses in fixed-sample designs where it is easy for the experimenter to fall in to the trap of assuming that the variances of the outcome distributions are too small, leading to underpowered studies.

Recent discussions both in the statistics literature and the scientific community at large has criticized the over-reliance on null hypothesis significance testing (e.g., \citealt{wasserstein_asa_2016,wasserstein_moving_2019,gelman_beyond_2014,greenland_for_2017,amrhein_scientists_2019}). Instead, many of these authors have argued that more emphasis should be placed on point estimation and estimates of uncertainty. In this paper, we take exactly such an approach where the two designs we propose, in contrast to much of the sequential testing literature, gives consistent estimators of the average treatment effect together with valid confidence intervals.

The rest of the paper is structured as follows. In Section 2, we introduce the theoretical framework we work in and make some preliminary definitions. In Section 3.1, we introduce the FWCID and prove both consistency and asymptotic efficiency of the naive version of that design. We also introduce the FPD and show how we can use the lessons from the FWCID to construct valid tests with correct power and size. In Section 3.2, we introduce a conservative version of the FWCID and prove consistency and asymptotic efficiency for that version as well. In Section 4, we study small sample behavior in a series of simulation studies where the FWCID and FPD are contrasted with some of the most common sequential designs used today. Finally, Section 5 concludes.

\section{Setup}\label{section:setup}

We consider a randomized experiment with treatment indicator $W\in \{0,1\}$. Let $Y_i(0)$ and $Y_i(1)$ denote the potential outcomes of unit $i$ if untreated and treated, respectively. The potential outcomes are independently and identically distributed across units with $\E(Y(0))=\mu_0$, $\E(Y(1))=\mu_1$, $0<\var(Y(0))=\sigma_0^2<\infty$ and $0<\var(Y(1))=\sigma_1^2<\infty$. 

The target parameter in the experiment is the population average treatment effect (ATE), $\tau = \E(Y(1)) - \E(Y(0))$. We assume that the stable unit treatment value assumption (SUTVA) holds, which means that we get the observed outcome, $Y_i$, as
\begin{equation}
    Y_i = W_i Y_i(1) + (1-W_i) W_i(0).
\end{equation}
We also assume exchangeability, which follows from randomization of $W$:
\begin{equation}
    Y(w) \ind W.
\end{equation}
For a sample of size $n$, the difference-in-means estimator can be written as
\begin{equation}
    \hat\tau_n = \frac{1}{n_1}\sum_{i=1}^n Y_i(1) W_i - \frac{1}{n_0}\sum_{i=1}^n Y_i(0) (1-W_i) = \hat\mu_{1n} - \hat\mu_{0n},
\end{equation}
where $n_0,n_1$ are the number of control and treated units, while $\hat\mu_{0n},\hat\mu_{1n}$ are the sample averages of $Y$ for the two groups. Considering the share treated as fixed, the variance of the difference-in-means estimator can be estimated as 
\begin{equation}
    \hat\sigma_{\hat\tau_n}^2 = \left(\frac{1}{n_1} \hat\sigma^2_{1n} + \frac{1}{n_0}\hat\sigma^2_{0n}\right),
\end{equation}
where $\hat\sigma^2_{wn}$ is the naive variance estimator for treated and non-treated respectively,
\begin{equation}
    \hat\sigma^2_{wn} = \frac{1}{n_w}\sum_{i:W_i=w} \left(Y_i(w)-\hat\mu_{wn}\right)^2.
\end{equation}
We focus on the naive variance estimator instead of the standard unbiased one, since our focus in on asymptotics and it will simplify the math later on. For completeness, we also set $\hat\sigma_{\hat\tau_n}^2=\infty$ if $n_0\leq 1$ and/or $n_1\leq 1$. It will be convenient to rewrite the variance in the following way: 
\begin{align}
    \hat\sigma_{\hat\tau_n}^2 &= \left(\frac{n_0}{n}\hat\sigma^2_{1n} + \frac{n_1}{n}\hat\sigma^2_{0n}\right)\kappa_n \nonumber \\
    &= \hat{\sigma}^2_{\mathcal{P}_n}\kappa_n, \label{eq:var_treatment_estimator}
\end{align}
where $\kappa_n:=1/n_0+1/n_1=n/(n_0n_1)$. In this paper, we focus on so-called \emph{sequential sampling designs}, i.e., designs where the sampling continues until a \emph{stopping rule} criterion is met.  
\begin{definition}
A sequential sampling design is a design where units enter sequentially such that $\lim_{n\rightarrow \infty} n_1/n = p \as$ , for $0<p<1$, and exchangeability and SUTVA holds for any $n\in \mathbb{N}^+$.
\end{definition}
As such, the treatment assignment does not have to be random, but could, for instance, be a design where every second unit goes to treatment and the others to control, as long as exchageablity and SUTVA holds. For any sequential design, the stopping rule, i.e., a rule for when to stop sampling, can be any function of data that satisfies the following rule:
\begin{definition}\label{def:stopping rule}
    A stopping rule $g(f(X_n), c_n)=N$ is a rule for a function $f$, a sequence of random variables $X_n$ and a sequence of constants $c_n$ such that sampling is stopped at sample size $n=N$ iff
    \begin{align}
        &f(X_n) > c_n, \quad \forall n < N, \nonumber \\
        &f(X_N) \leq c_n.
    \end{align}
\end{definition}
Note that $N$ is not the stopping rule, but the stochastic sample size at which sampling is stopped. Rather, the rule is to stop the first time $f(X_n)\leq c_n$.

\section{Fixed confidence interval width designs for ATE estimation}
In this section we define a sequential design that uses the observed confidence interval width as a stopping criterion. Formally,
we consider a fixed-width confidence interval as in Definition \ref{def:fixed_width_ci}.
\begin{definition}\label{def:fixed_width_ci}
    For a target parameter $\theta$ and an estimator $\hat\theta_N$, a fixed-width confidence interval with half-width $d$ and confidence level $1-\alpha$, $I_N$, is an interval
    \begin{equation}
        I_N = [\hat\theta_N - d, \hat\theta_N + d],
    \end{equation}
    such that 
    \begin{equation}
        \Pr(\theta \in I_N) = 1- \alpha.
    \end{equation}
\end{definition}
The difference from fixed-sample confidence intervals is that instead of the estimated variance being stochastic, it is now the sample size that is stochastic.

However, as we show in the following section, asymptotically, we can construct a sequence of fixed-sample confidence intervals and stop once such an interval has half-width smaller than $d$ to construct a valid fixed-width confidence interval. We call such a design a \emph{fixed-width confidence interval design} (FWCID).

\subsection{Asymptotically valid inference under FWCID}
\begin{theorem}
    For the target parameter $\tau$ and the stopping rule $g(\hat\sigma_{\hat\tau_n}^2,d^2/z_{\alpha/2}^2)=N$, with $0<\alpha<1$, and the estimator of the treatment effect $\hat\tau_N$, $I_N$ is an asymptotic fixed-width confidence interval. I.e., it is the case that 
    \begin{equation}
        \lim_{d\rightarrow 0} \Pr(\tau \in I_N) = 1-\alpha
    \end{equation}
    \label{thm:chow_robbins_style}
\end{theorem}
\begin{proof}
The proof follows almost directly from \cite{chow_asymptotic_1965} with some minor changes described below. Using Definition \ref{def:stopping rule}, the sample size $N$ is given by the following stopping rule
\begin{align}
   &\hat\sigma_{\hat\tau_n}^2 > d^2/z_{\alpha/2}^2, \quad \forall n < N, \nonumber \\
   &\hat\sigma_{\hat\tau_N}^2 \leq d^2/z_{\alpha/2}^2. \label{eq:stopping_rule}
\end{align}
For ease of notation, let $\sigma^2_\mathcal{P}:=(1-p)\sigma_1^2 + p\sigma_0^2$ and defining $\hat p_n = n_1/n$, we define the following variables:
\begin{align}
    s_n &:= \frac{\hat{\sigma}^2_{\mathcal{P}_n}}{\sigma^2_\mathcal{P}}\frac{p}{\hat p_n}\frac{1-p}{(1-\hat{p}_n)} \label{eq:def_sn} \\
    t &:= \frac{z_{\alpha/2}^2\sigma^2_\mathcal{P}}{d^2p(1-p)}. \label{eq:def_t}
\end{align}

The stopping rule in equation \eqref{eq:stopping_rule} can be rewritten as
\begin{equation}
N = \inf \{ n \in \mathbb{N}^+: s_n \leq n/t \}. \label{eq:def_N_v2}
\end{equation}
Note that $s_n>0\as$ and $\lim_{n\rightarrow\infty}  s_n = 1\as$ It follows that $N$ is a non-decreasing function of $t$ with $\lim_{t\rightarrow\infty} N = \infty\as$ Moreover, $\lim_{t\rightarrow \infty}$ is equivalent to $\lim_ {d\rightarrow 0}$.

By equation \eqref{eq:def_N_v2}, $s_N \leq N/t$ and $s_{N-1} > (N-1)/t$. The second inequality can be rewritten as $s_{N-1} N / (N-1) > N/t$. Hence, we have
\begin{equation}
    s_N \leq \frac{N}{t} < \frac{N}{N-1}s_{N-1}. \label{eq:ineq_dm}
\end{equation}
Because $\lim_{t\rightarrow\infty} N = \infty \as$, it is the case that $\lim_{t\rightarrow \infty} N/(N-1) = 1 \as.$ Finally, because $s_n>0 \as$ and $\lim_{n\rightarrow\infty}  s_n = 1$, from equation \eqref{eq:ineq_dm} it follows that
\begin{equation}
    1 = \lim_{t\rightarrow \infty} \frac{N}{t} \as = \lim_{d\rightarrow 0 }\frac{d^2N}{z_{\alpha/2}^2\sigma^2_\mathcal{P}}p(1-p) \as
    \label{eq:part1}
\end{equation}
We can write the probability that the interval covers $\tau$ as
\begin{align}
    \Pr(\tau \in I_N) &= \Pr(|\hat\tau_N - \tau| < d) \nonumber \\
    &= \Pr\left(\frac{\sqrt{N}\sqrt{p(1-p)}}{\sigma_\mathcal{P}}|\hat\tau_N - \tau| < \frac{\sqrt{N}\sqrt{p(1-p)}}{\sigma_\mathcal{P}}d \right)
    \label{eq:pr_ci}
\end{align}
By equation \eqref{eq:part1}, $\sqrt{N}\sqrt{p(1-p)}d/\sigma_\mathcal{P}$ converges to $z_{\alpha/2}$ in probability and $N/t$ converges to 1 in probability as $t\rightarrow \infty$. By \cite{anscombe_large-sample_1952}, it follows that as $t\rightarrow \infty$, $\sqrt{N}(\hat\tau_N - \tau)\sqrt{p(1-p)}/\sigma_\mathcal{P}\sim \mathcal{N}(0,1)$ Hence, it is the case that, as $t\rightarrow \infty$, $\Pr(\tau \in I_N) = F(-z_{\alpha/2}) - F(z_{\alpha/2})$, which concludes the proof.
\end{proof}
Theorem \ref{thm:chow_robbins_style} says that as we let the target interval width go to zero, the coverage for the CI at the stopping point is the intended $1-\alpha$. Moreover, the difference-in-means estimator is consistent under FWCID, formalized in Corollary \ref{corollary:tauhat_consistent}. 
\begin{corollary}\label{corollary:tauhat_consistent}
Consider a sequential experiment with the same stopping rule as in Theorem \ref{thm:chow_robbins_style}. The difference-in-means estimator is a consistent estimator of the average treatment effect, i.e., $\hat{\tau}_N \xrightarrow{p} \tau \,\,\text{as} \,\, d \rightarrow 0.$
\end{corollary}
\begin{proof}
By Theorem \ref{thm:chow_robbins_style}, $\hat{\tau}_N$ is asymptotically normal as $d\rightarrow0$. The corollary follows immediately, see e.g. \cite{casella2021statistical}.     
\end{proof}
For readers familiar with the sequential testing literature, both Theorem \ref{thm:chow_robbins_style} and Corollary \ref{corollary:tauhat_consistent} might be somewhat surprising. In the sequential testing literature, the goal is to mitigate inflated false positive rates due to ``peeking'' at the outcome data during the sampling. Moreover, using sequential tests also gives biased point estimators \citep{fan2004conditional}. However, in traditional sequential testing, the treatment effect estimator is directly used in the stopping rule, e.g, stopping when the treatment effect is significantly different from zero. Theorem \ref{thm:chow_robbins_style} and Corollary \ref{corollary:tauhat_consistent} instead says that using information about the variance of the estimator as basis for a stopping rule, as opposed to the ATE estimator itself, does not affect the coverage and the estimator remains consistent. The result directly implies that the false positive rate of the corresponding test (i.e., whether the interval covers the null or not) is not inflated, something that we discuss further in Section \ref{sec:hypothesis_testing}.

\begin{remark}\label{rem:as_eff}
FWCID is asymptotically efficient in the sense that the asymptotic sample size is the same for FWCID as for a fixed-sample design. To see this, let $d_f$ be the half-width of a fixed-sample confidence interval. It is the case that
\begin{equation}
    d_f = z_{\alpha/2} \sqrt{\frac{\hat\sigma^2_{\mathcal{P}_n}}{n\hat{p}_n(1-\hat{p}_n)}},
\end{equation}
which means that
\begin{equation}
    \lim_{n\rightarrow \infty} \frac{d_f^2 n}{z_{\alpha/2}\sigma^2_\mathcal{P}}p(1-p) = 1 \as
\end{equation}
Hence, by equation \eqref{eq:part1}, it is the case that
\begin{equation}
    \lim_{n\rightarrow \infty} d_f^2 n = \lim_{d\rightarrow 0 }d^2N \as.
\end{equation}
This result says that, asymptotically, the number of observations needed for a given confidence interval width is the same for FWCID as for a standard fixed-sample design.\footnote{In fact, it is also possible to show asymptotic efficiency in the sense of \citealt{chow_asymptotic_1965}, which is that $\lim_{d\rightarrow 0} \frac{d^2\E(N)}{a^2\sigma^2_\mathcal{P}}p(1-p)=1$. We omit a proof of this result here, but it follows directly from \citealt{chow_asymptotic_1965}}. 
\end{remark}

\begin{remark} \label{rem:equiv_stopping_rules}
    The stopping rule $g(\hat\sigma_{\hat\tau_n}^2,d^2/z_{\alpha/2}^2)=N$ in Theorem \ref{thm:chow_robbins_style} implies that sampling is stopped when the variance of the treatment effect is sufficiently small. By noting that $\hat\sigma_{\hat\tau_n}^2 = \hat{\sigma}^2_{\mathcal{P}_n} /(n\hat{p}_n(1-\hat{p}_n))$, the stopping rule can be rewritten as
    \begin{equation}
        g\left(\frac{\hat{\sigma}^2_{\mathcal{P}_n} z_{\alpha/2}^2}{d^2\hat{p}_n(1-\hat{p}_n)},n\right)=N. \label{eq:alt_stopping_rule}
    \end{equation}
\end{remark}
The advantage of writing the stopping rule this way is that the first argument provides an estimate of when sampling will be stopped. I.e., let $\hat{N}_n$ be the estimated required sample size to get a confidence interval with half-width $d$, we have
\begin{equation}
    \hat{N}_n = \frac{\hat{\sigma}^2_{\mathcal{P}_n} z_{\alpha/2}^2}{d^2\hat{p}_n(1-\hat{p}_n)}.
\end{equation}
For most designs, $p$ would be known, in which case, $\hat{p}$ could be replaced with $p$ (something that would mainly be useful early on in the experiment when $\hat{p}$ is likely to be volatile).

The two equivalent stopping rules are illustrated in Figure \ref{fig:stop_rule_illustration} for some randomly generated data in a Bernoulli trial. In the top diagram, the stopping rule in Theorem \ref{thm:chow_robbins_style} is illustrated where $\hat\sigma^2_{\hat\tau_n}$ is generally decreasing, with the decreases becoming smoother over time as each individual observation becomes relatively less important. In the bottom diagram, the alternative stopping rule in equation \eqref{eq:alt_stopping_rule} is instead illustrated. Here the expected sample size at which the experiment is stopped does not trend downwards, but is converging towards a finite value.

\begin{figure}
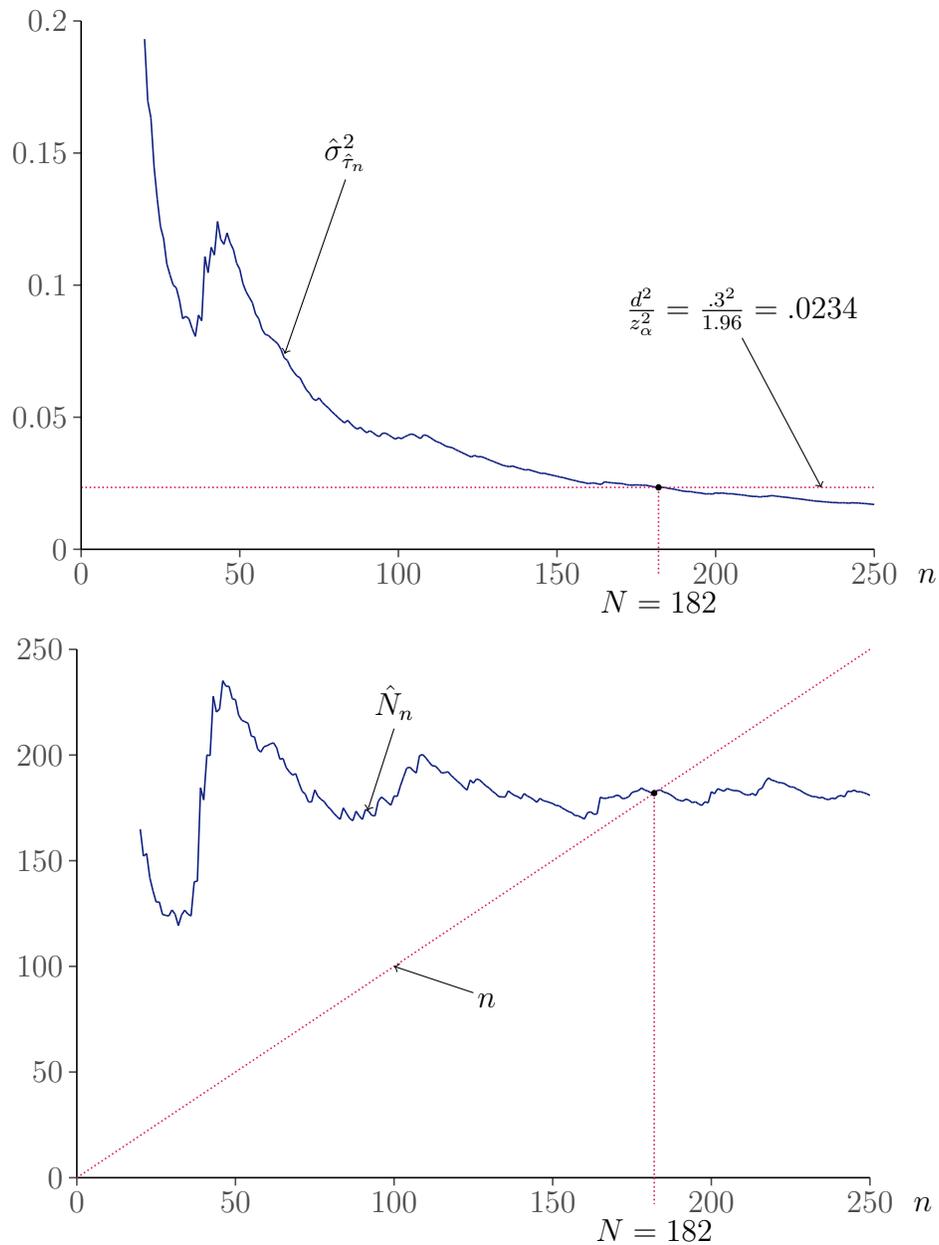

    \centering
    \input{figures/fig_stop_rule1} \\
    \input{figures/fig_stop_rule2} \\
    \caption{Illustration of stopping rules}
    \label{fig:stop_rule_illustration}
    \floatfoot{Note: The figure illustrates the two equivalent stopping rules $g(\hat\sigma_{\hat\tau_n}^2,d^2/z_{\alpha/2}^2)=N$ (top diagram) and $ g((\hat{\sigma}^2_{\mathcal{P}_n} z_{\alpha/2}^2)/(d^2\hat{p}_n(1-\hat{p}_n)),n)=N$ (bottom diagram). Both $Y(0)$ and $Y(1)$ have been drawn from a standard normal distribution, and treatment assignment is determined in a Bernoulli trial with $p=.5$. The half-width of the confidence interval is set at $d=.3$ and the confidence level is 95\%.}
\end{figure}

\begin{remark}\label{rem:unbiasedness}
    In the special case where $\E(\hat{p}_n)=1/2$ and constant treatment effects (i.e., $Y_i(1)-Y_i(0)=\tau$ for all $i$) and $\hat{p}_n \indep Y_i(w)$, it is the case that $\E(\hat\tau_N)=\tau$, i.e., the treatment effect estimator is unbiased.
\end{remark}
A sufficient condition for $\E(\hat\tau_N)=\tau$ is that $\E(\hat\tau_n|\hat\sigma_{\hat\tau_n}^2) =  \E(\hat\tau_n) = \tau$ for all $n$. We have
\begin{align}
    \E(\hat\tau_n|\hat\sigma_{\hat\tau_n}^2) &= \E\left(\hat\tau_n|\hat\sigma_{\hat\tau_n}^2\right) \nonumber \\
    &=\E\left(\hat\mu_{1n}-\hat\mu_{0n}|\hat\sigma_{\hat\tau_n}^2\right)  \nonumber
    \\
    &=\E\left(\hat\mu_{1n}|\hat\sigma_{\hat\tau_n}^2\right) -\E\left(\hat\mu_{0n}|\hat\sigma_{\hat\tau_n}^2\right) \nonumber \\
    &=\E\left(\hat\mu_{1n}\bigg|\frac{1}{\hat{p}_n n}\hat\sigma^2_{1n} + \frac{1}{(1-\hat{p}_n) n}\hat\sigma^2_{0n}\right) -\E\left(\hat\mu_{0n}\bigg|\frac{1}{\hat{p}_n n}\hat\sigma^2_{1n} + \frac{1}{(1-\hat{p}_n) n}\hat\sigma^2_{0n}\right)
     \nonumber \\
    &=\E\left(\hat\mu_{1n}\bigg|\frac{1}{\hat{p}_n n}\hat\sigma^2_{1n} \right) -\E\left(\hat\mu_{0n}\bigg|\frac{1}{(1-\hat{p}_n) n}\hat\sigma^2_{0n}\right)
    \nonumber \\
    &=\E\left(\hat\mu_{1n}\bigg|\frac{1}{\hat{p}_n n}\hat\sigma^2_{1n} \right) -\E\left(\hat\mu_{0n} + \tau \bigg|\frac{1}{(1-\hat{p}_n) n}\hat\sigma^2_{0n}\right) + \tau.
\end{align}
Because of the constant treatment effect and the fact that $\E(\hat{p}_n)=1/2$, the  conditional distribution of $\hat\mu_{1n}|(1/(\hat{p}_n n)\hat\sigma^2_{1n})$ equals the conditional distribution of $\hat\mu_{0n}|(1/((1-\hat{p}_n) n))\hat\sigma^2_{0n}$, and we have $\E(\hat\tau_n|\hat\sigma_{\hat\tau_n}^2)=\tau$ for all $n$.

\subsubsection{Implications for hypothesis testing \label{sec:hypothesis_testing}}

Theorem \ref{thm:chow_robbins_style} implies that the experimenter decides on a confidence half-width, $d$, together with a confidence level, $\alpha$, and let the experiment run until a the stopping rule is satisfied. Alternatively, instead of deciding on a confidence width, the experimenter may consider a null hypothesis, $\tau = \tau_{H_0}$ and a hypothesized alternative, $\tau_{H_1}$ for which a given power, $1-\beta$, should be achieved. For a one-sided hypothesis test, we want to find a stopping rule such that the null is rejected with probability $\alpha$ if $\tau = \tau_{H_0}$ and with probability $1-\beta$ if $\tau = \tau_{H_1}$. We call such a design a \emph{fixed-power design} (FPD). We formalize this design in Corollary \ref{cor:equiv_d_power_stopping_rule}.

\begin{corollary}\label{cor:equiv_d_power_stopping_rule}
    For a null hypothesis of $\tau_{H_0}$ and an alternative $\tau_{H_1}$, it is the case that for $\tau_{d}:=\tau_{H_1}-\tau_{H_0}$, the stopping rule $g(\hat\sigma_{\hat\tau_n}^2,\tau_{d}^2/(z_{\alpha}+z_\beta)^2)=N$ implies that
    \begin{align}
    &\lim_{\tau_{d} \rightarrow 0}\Pr( \hat\tau_N - z_{\alpha}\hat\sigma_{\hat\tau_N}^2 > \tau_{H_0}| \tau=\tau_{H_0}) = \alpha, \label{eq:null_one_sided} \\
    &\lim_{\tau_{d} \rightarrow 0}\Pr( \hat\tau_N - z_{\alpha}\hat\sigma_{\hat\tau_N}^2 > \tau_{H_0}| \tau=\tau_{H_1}) = 1-\beta, \label{eq:alt_one_sided}
    \end{align}
    for $\tau_{H_1} > \tau_{H_0}$ and
    \begin{align}
    &\lim_{\tau_{d} \rightarrow 0}\Pr( \hat\tau_N + z_{\alpha}\hat\sigma_{\hat\tau_N}^2 < \tau_{H_0}| \tau=\tau_{H_0}) = \alpha, \\
    &\lim_{\tau_{d} \rightarrow 0}\Pr( \hat\tau_N + z_{\alpha}\hat\sigma_{\hat\tau_N}^2 < \tau_{H_0}| \tau=\tau_{H_1}) = 1-\beta,
    \end{align}
    for $\tau_{H_1} <\tau_{H_0}$.
\end{corollary}
\begin{proof}
    We prove this corollary for $\tau_{H_1} > \tau_{H_0}$, but where the exact same steps can be taken to prove the result for $\tau_{H_1} < \tau_{H_0}$. We begin by noting that equation \eqref{eq:null_one_sided} follows directly from Theorem \ref{thm:chow_robbins_style}, so what remains to be proven is equation \eqref{eq:alt_one_sided}. Theorem \ref{thm:chow_robbins_style} implies that for a given half-width of a confidence interval, $d$, as $d \rightarrow 0$, it is the case that $\hat\tau \sim \mathcal{N}(\tau, d^2/z^2)$, for the stopping rule $g(\hat\sigma_{\hat\tau_n}^2,d^2/z^2)=N$ with $z>0$, which means that 
    \begin{equation}
        \lim_{d \rightarrow 0} \Pr( (\hat\tau_N - \tau_{H_1}) / \hat\sigma_{\hat\tau_N} < z  | \tau=\tau_{H_1}) = F(z).
    \end{equation}
    For finite $z$, together with the stopping rule which implies that $\tau_d=d$, we get
    \begin{equation}
        \lim_{\tau_d \rightarrow 0} \Pr( (\hat\tau_N - \tau_{H_1}) / \hat\sigma_{\hat\tau_N} < z  | \tau=\tau_{H_1}) = F(z).
    \end{equation}
    We can re-write this equation as
    \begin{equation}
        \lim_{\tau_d \rightarrow 0} \Pr( \hat\tau_N - z\hat\sigma_{\hat\tau_N}< \tau_{H_1}  | \tau=\tau_{H_1}) = F(z).
    \end{equation}
    Using the fact that $\tau_{H_0}=\tau_{H_1} - (z_{\alpha}+z_\beta)\hat\sigma_{\hat\tau_N}$, we get
    \begin{equation}
        \lim_{\tau_d \rightarrow 0} \Pr( \hat\tau_N - (z + z_{\alpha}+z_\beta)\hat\sigma_{\hat\tau_N}< \tau_{H_0}  | \tau=\tau_{H_1}) = F(z). \label{eq:Fz}
    \end{equation}
    Let $z=-z_\beta$ and switching the inequality sign, we get
    \begin{equation}
        \lim_{\tau_d \rightarrow 0} \Pr( \hat\tau - z_{\alpha}\hat\sigma_{\hat\tau_N}> \tau_{H_0}  | \tau=\tau_{H_1}) = F(z_\beta).
    \end{equation}
    By definition, $F(z_\beta)=1-\beta$, which proves the result.
\end{proof}
Corollary \ref{cor:equiv_d_power_stopping_rule} is a powerful result. It says that for a FPD to have size $\alpha$ and power $1-\beta$, it is sufficient to specify $\tau_{H_0}$, $\tau_{H_1}$ together with $\alpha$ and $\beta$. This is in contrast to traditional fixed-sample designs where it is also necessary to specify $\sigma_0^2$ and $\sigma_1^2$ (i.e., the variance of the outcome under control and treatment), something that is often the most difficult part of a power analysis. In some settings, estimates of $\sigma_0^2$ may be available with pre-experimental data. However, $\sigma_1^2$ may be more difficult to assess if no one historically has received the treatment, something that would be the case in many clinical trial settings, as well as in A/B testing in the tech industry. 

Still, the FPD is not a ``free lunch'' because it comes with one drawback compared to fixed-sample designs: the sample size, $N$, is stochastic. This fact could be a problem in settings where sampling is costly, and the experimenter does not want to run the risk of sampling for ``too long''. As noted in Remark \ref{rem:equiv_stopping_rules} for the FWCID, the stopping rule can be rewritten such that---at any point during the experiment---a forecast for how long the experiment will run for can be made as
\begin{equation}
    \hat{N}_n = \frac{\hat{\sigma}^2_{\mathcal{P}_n} (z_{\alpha} + z_\beta)^2}{(\tau_{H_1}-\tau_{H_0})^2\hat{p}_n(1-\hat{p}_n)},
\end{equation}
where $\hat{N}_n$ is interpreted as the estimated required sample size to achieve power of $1-\beta$ if $\tau=\tau_{H_1}$. Hence, if the forecast made suggest that the experiment would need to run for longer than what is defensible for ethical, monetary or other reason, the experiment can be aborted and the resulting fixed-sample confidence interval at that point can be used as confidence interval. Note, of course, that this only works if the decision to terminate is based on $\hat{N}_n$ (thereby implicitly on $\hat{\sigma}^2_{\mathcal{P}_n}$) and not based on $\hat\tau_n$.

Finally, we also note that we can get a double-sided test with correct size $\alpha$ and approximate power $1-\beta$ by replacing $z_\alpha$ with $z_{\alpha/2}$.\footnote{The power is only approximate because is excludes the very small probability of rejecting the null in the ``wrong direction'' (i.e., rejecting the null because $\hat\tau_N$ is sufficiently smaller than $\tau_{H_0}$ when $\tau_{H_1} > \tau_{H_0}$), something that is not relevant in practice. It is possible to get the correct power by noting that $F(z_\beta') + F(-(2z_{\alpha/2} + z_\beta')) = 1-\beta$ for some $z_\beta'$; a value that is straightforward to solve for numerically.}



\subsection{Asymptotically conservative FWCID \label{sec:consFWCID}}
Theorem \ref{thm:chow_robbins_style} and Corollary \ref{cor:equiv_d_power_stopping_rule} gives asymptotic guarantees that FWCID gives correct confidence interval coverage and that FPD gives correct size and power. For finite samples, the coverage of the confidence interval under FWCID will generally be too small since the stopping rule implies that the estimated variance of the treatment effect, $\hat\sigma_{\hat\tau_N}^2$ is (slightly) smaller, in expectation, than the true variance, $\sigma_{\hat\tau_N}^2$. This is intuitive, as we will stop too early only when we underestimate the variance, but never when we overestimate it. Another way of framing this is to say that since $\hat\sigma_{\hat\tau_N}^2 = \hat\sigma_{\mathcal{P}_N}^2 \kappa_N$, for $d>0$, it is generally the case that $\E(\hat\sigma_{\mathcal{P}_N}^2)<\sigma_{\mathcal{P}}^2$. 

To deal with this issue, we would like to find a conservative estimator for $\sigma_{\mathcal{P}}^2$. The way we propose to achieve this is to form a upper confidence sequence for $\sigma_{\mathcal{P}}^2$ that is ``always-valid'' (also called ``anytime valid''). This sequence, $U_n$, for a target parameter $\theta$ is an always-valid
$(1 - \alpha)$ upper confidence sequence if, given any
stopping time $n$, it holds that:
\begin{equation}
    \Pr(\theta \leq U_n) \geq 1 - \alpha, \forall n \in \mathbb{N}^+,
\end{equation}
see, e.g., \cite{howard2021}. Here we rely on the asymptotic version of such a confidence sequence as defined in \cite{waudby2021time}.
The upper bound of the always-valid confidence sequence forms a conservative sequence of estimates for $\sigma_{\mathcal{P}}^2$ which can be used in the stopping rules in Theorem \ref{thm:chow_robbins_style} and Corollary \ref{cor:equiv_d_power_stopping_rule}. Lemma \ref{lemma:ci_for_pooled_var} shows how such a sequence can be formed.

\begin{lemma}\label{lemma:ci_for_pooled_var}
    $\hat{\sigma}^2_{\mathcal{P}_n} + \sqrt{\var(Z)} \cdot \phi_n(\rho,\alpha_c)$,
    forms an asymptotic $1-\alpha_c$ always-valid upper confidence sequence for $\sigma^2_{\mathcal{P}}$, where
    \begin{equation}
        \phi_n(\rho,\alpha_c) := \sqrt{\frac{2(n\rho_c^2+1)}{n^2\rho^2}\log\left(1+\frac{\sqrt{n\rho^2+1}}{2\alpha_c}\right)},
    \end{equation}
    for the uncentered influence function $Z_i := \frac{1-p}{p}W_i (Y_i - \mu_{1})^2 + \frac{p}{1-p}(1 - W_i) (Y_i - \mu_{0})^2$ and any constant $\rho> 0$.
\end{lemma}
\begin{proof}
By Corollary 3.4 and Proposition B.1 in \cite{waudby2021time}, using the fact that $\var(Z - \sigma^2_\mathcal{P}) = \var(Z)$, Lemma \eqref{lemma:ci_for_pooled_var} holds if $\var(Z)$ is finite and
\begin{align}
    \hat{\sigma}^2_{\mathcal{P}_n} - \sigma^2_{\mathcal{P}} = \frac{1}{n}\sum_{i=1}^n (Z_i-\sigma^2_{\mathcal{P}}) + o\left(\sqrt{\frac{\log(n)}{n}} \right), \as \label{eq:eq48waudby}
\end{align}
Define $\tilde{\sigma}_{wn}^2 := \frac{1}{n_w}\sum_{i:W_i=w}^{n}(Y_i(w)-\mu_w)^2$, for $w=0,1$  and $\tilde{\sigma}^2_{\mathcal{P}_n}:=\frac{1}{n} \sum_{i=1}^n Z_i$ as natural variance estimators for $\sigma_w^2$ and $\sigma^2_\mathcal{P}$ if $\mu_w$ had been known. Equation \eqref{eq:eq48waudby} simplifies to
\begin{align}
    \hat{\sigma}^2_{\mathcal{P}_n} -  \tilde{\sigma}^2_{\mathcal{P}_n} &= 
    \frac{n_0}{n}\hat\sigma^2_{1n} + \frac{n_1}{n}\hat\sigma^2_{0n} -\frac{1}{n} \sum_{i=1}^n Z_i +o\left(\sqrt{\frac{\log(n)}{n}} \right)\nonumber \\ 
    &= 
    (1-\hat{p}_n)\hat\sigma^2_{1n} + \hat{p}_n\hat\sigma^2_{0n} - (1-p)\frac{\hat{p}_n}{p} \tilde{\sigma}^2_{1n} - p\frac{1-\hat{p}_n}{1-p} \tilde{\sigma}^2_{0n}+o\left(\sqrt{\frac{\log(n)}{n}} \right)
    \nonumber \\ 
    &=(1-\hat{p}_n) \left(\hat\sigma^2_{1n} -  \frac{1-p}{1-\hat{p}_n}\frac{\hat{p}_n}{p} \tilde{\sigma}^2_{1n}\right) + 
    \hat{p}_n \left(\hat\sigma^2_{0n} -  \frac{p}{\hat{p}_n}\frac{1-\hat{p}_n}{1-p} \tilde{\sigma}^2_{0n}\right) + o\left(\sqrt{\frac{\log(n)}{n}} \right).
\end{align}

Because of symmetry of $ \hat{p}_n\left(\hat{\sigma}^2_{1n}-\tilde{\sigma}^2_{1n}\right)$ and $(1-\hat{p}_n)\left(\hat{\sigma}^2_{0n}-\tilde{\sigma}^2_{0n}\right)$, it is sufficient to show that one of the first two terms is $o(\sqrt{\log(n)/n}) \as$ for equation \eqref{eq:eq48waudby} to hold. Define $v_n:=(1-p)\hat{p}_n/((1-\hat{p}_n)p)$, we have
\begin{align}
   \hat{p}_n(\hat{\sigma}^2_{1n}-v_n\tilde{\sigma}_{1n}^2) &= \hat{p}_n\frac{1}{n_1}\sum_{i:W_i=1} \left(Y_i(1)-\hat\mu_{1n}\right)^2-\hat{p}_nv_n\frac{1}{n_1}\sum_{i:W_i=1} \left(Y_i(1)-\mu_{1} \right)^2 \nonumber \\
   &= \hat{p}_n\left(\frac{1}{n_1} \sum_{i:W_i=1}Y_i(1)^2 - \hat{\mu}_{1n}^2 \right) - \hat{p}_nv_n\left(\frac{1}{n_1} \sum_{i:W_i=1}Y_i(1)^2 + \mu_1^2 - 2 \hat\mu_{1n} \mu_1\right)\nonumber \\
   &=-\hat{p}_n (\hat\mu_{1n}^2 + v_n\mu_1^2 - v_n2 \hat\mu_{1n}\mu_1) +(1-v_n) \hat{p}_n\frac{1}{n_1} \sum_{i:W_i=1}Y_i(1)^2\nonumber\\
   &=-\hat{p}_n (\hat\mu_{1n}^2 + \mu_1^2 - 2 \hat\mu_{1n}\mu_1) + \hat{p}_n(1-v_n)\left(\mu_1^2 - 2 \hat\mu_{1n}\mu_1+\frac{1}{n_1} \sum_{i:W_i=1}Y_i(1)^2 )\right)\nonumber\\
   &=-\hat{p}_n (\hat{\mu}_{1n}-\mu_{1})^2 + \hat{p}_n(1-v_n)\left(\mu_1^2 - 2 \hat\mu_{1n}\mu_1+\frac{1}{n_1} \sum_{i:W_i=1}Y_i(1)^2 )\right).
\end{align}
The second term is $O(1/\sqrt{n}) \cdot o(1) \as$, and so we only need to verify that the first term is $o(\sqrt{\log(n)/n})$.
Studying the absolute difference, $\hat{p}_n |(\hat{\mu}_{1n}-\mu_{1})|$ and using the  triangle inequality, we have
\begin{align*}
    \hat{p}_n|\hat{\mu}_{1n}-\mu_1| &= \left|\hat{p}_n\hat{\mu}_{1n} - \hat{p}_n\mu_1\right|\\
     &= \left|\hat{p}_n\hat{\mu}_{1n} - p\mu_1 - (\hat{p}_n\mu_1 - p\mu_1)\right|\\
    &\leq \left|\hat{p}_n\hat{\mu}_{1n} - p\mu_1 \right| - \mu_1\left|\hat{p}_n-p\right|.
\end{align*}
Each of these two terms are mean-centered i.i.d random variables which, by LIL (law of the iterated logarithm), means that
\begin{equation}
    \hat{p}_n|\hat{\mu}_{1n}-\mu_1| = O\left(\sqrt{\frac{\log(\log(n))}{n}}\right), \as \Longleftrightarrow \hat{p}_n(\hat{\mu}_{1n}-\mu_1)^2=O\left(\frac{\log(\log(n))}{n}\right), \as
\end{equation}
Hence,
\begin{equation}
    |\hat{\sigma}^2_{\mathcal{P}_n}-\tilde{\sigma}^2_{\mathcal{P}_n} | = O\left(\frac{\log(\log(n))}{n}\right) = o\left(\sqrt{\frac{\log(n)}{n}}\right), \as,
\end{equation}
which concludes the proof.
\end{proof}
Lemma \ref{lemma:ci_for_pooled_var} includes the unknown variance of $Z$. By replacing $\mu_w$ with $\hat\mu_w$ we can define
\begin{equation}
    \hat{Z}_i := W_i(Y_i(1)-\hat\mu_{1})^2 + (1-W_i)(Y_i(0)-\hat\mu_{0})^2,
\end{equation}
and estimate the variance as
\begin{equation}
    \widehat{\var(Z)} = \frac{1}{n}\sum_{i=1}^n\left(\hat{Z}_i - \frac{1}{n} \sum_{j=1}^n \hat{Z}_j\right)^2 =
    \frac{1}{n}\sum_{i=1}^n\left(\hat{Z}_i - \hat{\sigma}^2_{\mathcal{P}_n}\right)^2.
\end{equation}

Using Lemma \ref{lemma:ci_for_pooled_var} we can construct an asymptotic always-valid upper confidence sequence for $\sigma^2_{\mathcal{P}}$ for a given confidence level $\alpha_c$. The upper bound of this confidence sequence is
\begin{equation}
    \hat\sigma^2_{\mathcal{P}_n^{ub}} := \hat{\sigma}^2_{\mathcal{P}_n} +  \sqrt{\widehat{\var(Z)}} \cdot \phi_n(\rho,\alpha_c).\label{eq:def_sigma2_p_ub}
\end{equation}
Using this result means that we can now construct a conservative stopping rule as
\begin{equation}
    g( \hat\sigma^2_{\mathcal{P}_n^{ub}}\kappa_n, d^2/z_{\alpha/2}^2) = N^c. \label{eq:cons_stopping_rule}
\end{equation}
I.e., $N^c$ is the sample size at which sampling is stopped. We can now state the analogous result to Theorem \ref{thm:chow_robbins_style} for the stopping rule in equation \eqref{eq:cons_stopping_rule}:
\begin{theorem}
    Assume that $\E(Y(1)^4)<\infty$ and $\E(Y(0)^4)<\infty$. For the target parameter $\tau$ and the stopping rule $g( \hat\sigma^2_{\mathcal{P}_n^{ub}}\kappa_n, d^2/z_{\alpha/2}^2) = N^c$, with $0<\alpha<1,0<\alpha_c<1$, and the estimator of the treatment effect $\hat\tau_{N^c}$, $I_{N^c}$ is a fixed-width confidence interval. I.e., it is the case that 
    \begin{equation}
        \lim_{d\rightarrow 0} \Pr(\tau \in I_{N^c}) = 1-\alpha
    \end{equation}
    \label{thm:chow_robbins_style_cons}
\end{theorem}
\begin{proof}
    The proof is analogous to the proof of Theorem \ref{thm:chow_robbins_style}, but where $s_n$ in equation \eqref{eq:def_sn} is replaced with:
    \begin{equation}
        s_n^c := \frac{\hat\sigma^2_{\mathcal{P}_n^{ub}}}{\sigma^2_{\mathcal{P}}}\frac{p}{\hat{p}_n}\frac{1-p}{1-\hat{p}_n}.
    \end{equation}
    Once again, $s_n^c>0 \as$, and so, if we can show that
    \begin{equation}
        \lim_{n\rightarrow\infty}  s_n^c = 1, \as, \label{eq:cond_thm2}
    \end{equation}
    then the theorem is proven by following the exact same steps as in the proof for Theorem \ref{thm:chow_robbins_style}. Furthermore, using Lemma \ref{lemma:ci_for_pooled_var} (which holds because of finite fourth moments), equation \eqref{eq:cond_thm2} is true if
    \begin{equation}
        \lim_{n\rightarrow\infty}
        \frac{\hat\sigma^2_{\mathcal{P}_n^{ub}}}{\sigma^2_{\mathcal{P}}} =
        1, \as
    \end{equation}
    By the definition of $\hat\sigma^2_{\mathcal{P}_n^{ub}}$ in equation \eqref{eq:def_sigma2_p_ub}, together with the fact that $\lim_{n\rightarrow\infty}\hat\sigma^2_{\mathcal{P}_n}/\sigma^2_{\mathcal{P}} = 1 \as$, we have
    \begin{equation}
        \lim_{n\rightarrow\infty}
        \frac{\hat\sigma^2_{\mathcal{P}_n^{ub}}}{\sigma^2_{\mathcal{P}}} =
        1 + \frac{\sqrt{\frac{1}{n}\sum_{i=1}^n\left(\hat{Z}_i - \hat{\sigma}^2_{\mathcal{P}_n}\right)^2} \cdot \phi_n(\rho,\alpha_c)}{\sigma^2_{\mathcal{P}}}.
    \end{equation}
    Hence, we need to show that the second term tends to zero almost surely. It is the case that $\phi_n(\rho,\alpha_c)=o(1)$ for $\alpha_c>0$. We also have
    \begin{multline}
        \frac{1}{n}\sum_{i=1}^n\left(\hat{Z}_i - \hat{\sigma}^2_{\mathcal{P}_n}\right)^2 = \frac{1}{n}\sum_{i=1}^n \hat{Z}_i^2 -  \hat{\sigma}^4_{\mathcal{P}_n} = \\
          \frac{1}{n}\left(\sum_{i=1}^n W_i(Y_i(1)-\hat\mu_1)^4 + \sum_{i=1}^n (1-W_i)(Y_i(0)-\hat\mu_0)^4 \right)- \hat{\sigma}^4_{\mathcal{P}_n}.
    \end{multline}
    Let $\mu_{4w}$ be the fourth central moment for $Y(w)$, it is the case that
    \begin{equation}
         \lim_{n\rightarrow\infty}
         \frac{1}{n}\sum_{i=1}^n\left(\hat{Z}_i - \frac{1}{n} \sum_{j=1}^n \hat{Z}_j\right)^2 =p\mu_{41} + (1-p)\mu_{40} - \sigma^4_{\mathcal{P}} \as 
    \end{equation}
    Because we assumed finite fourth moments of the potential outcomes, this term is $O(1) \as$ Hence,
    \begin{equation}
        \lim_{n\rightarrow\infty}
        \frac{\hat\sigma^2_{\mathcal{P}_n^{ub}}}{\sigma^2_{\mathcal{P}}} = 1 + O(1) \cdot o(1) \as = 1 \as,
    \end{equation}
    which proves the theorem.
\end{proof}
Just as the naive FWCID design in Theorem \ref{thm:chow_robbins_style} gives the FPD design in Corollary \ref{cor:equiv_d_power_stopping_rule}, we can use the conservative FWCID design in Theorem \ref{thm:chow_robbins_style_cons} to get a conservative version of the FPD design \ref{cor:equiv_d_power_stopping_rule} by replacing $\hat\sigma_{\hat\tau_n}^2$ with $\hat\sigma^2_{\mathcal{P}_n^{ub}} \kappa_n$ in Corollary \ref{cor:equiv_d_power_stopping_rule}.

It is important to distinguish between $\alpha$ and $\alpha_c$. The former gives the confidence level of the targeted confidence interval for the treatment effect estimator. The latter is the confidence level of the confidence sequence of $\sigma^2_{\mathcal{P}}$. These two values need not be the same and $\alpha_c$ should be chosen based on just how conservative the experimenter wants to be (or, phrased differently, how much to insure against stopping too early), where a smaller $\alpha_c$ means being more conservative (stopping later).

However, regardless of how small a value of $\alpha_c$ is chosen, asymptotically, any $\alpha_c>0$ leads to asymptotic efficiency in the sense discussed in Remark \ref{rem:as_eff}. I.e., the ratio of $N$ and $N^c$ converges to 1 as $d\rightarrow 0$. Of course, for finite-sample performance, the choice of $\alpha_c$ matters.
One may question why we are not simply using the always-valid sequential confidence sequence directly for the treatment effect estimator, using the results in, e.g., \cite{howard2021}. Doing so would clearly by definition make the CI coverage bounded by $1-\alpha$ at any point, including the stopping point given by a CI-width based rule. The reason is that these always-valid confidence sequences are less efficient. Instead of converging at the rate $O(1/\sqrt{n})$, these sequences instead converge at the rate implied by LIL, i.e.,  $O(\log (\log( n)) / \sqrt{n})$, which means that the ratio of the sample size implied by such a design to $N$ or $N^c$ will tend to infinity asymptotically. We show the poor finite-sample performance (in terms of efficiency) of this design in the simulations below.

\section{Simulation Study}
To study finite sample behavior of the FWCID, we perform Monte Carlo studies. In addition to naive (Theorem \ref{thm:chow_robbins_style}) and conservative (Theorem \ref{thm:chow_robbins_style_cons}) versions of the FWCID, we also include an always-valid version using the interval in \cite{waudby2021time}. That is, sampling is stopped once the always-valid half-width is smaller than $d$. The half-width is equal to $v_n \cdot \psi_n(\rho,\alpha)$, where these values are given by (see Appendix A.1 in \citealt{maharaj2023anytime}):
\begin{align}
    v_n &= \sqrt{\frac{1}{\hat{p}_n}(\hat\sigma^2_{1n} + \hat\mu_{1n}^2) + \frac{1}{1-\hat{p}_n}(\hat\sigma^2_{0n} + \hat\mu_{0n}^2)- (\hat\mu_{1n}-\hat\mu_{0n})^2}, \\
    \psi_n(\rho,\alpha) &= \sqrt{\frac{2(n\rho^2+1)}{n^2\rho^2} \log\left(\frac{\sqrt{n\rho^2+1}}{\alpha}\right)}.
\end{align}
Note that $\psi_n(\rho,\alpha)$ differs from $\phi_n(\rho,\alpha)$ because the former is for a symmetric confidence interval, whereas the latter is for a upper confidence interval (i.e., the lower bound is set to $-\infty$). To summarize, the stopping rules for the three designs we consider are:
\begin{align*}
    \text{FWCID, naive}: & \quad g(\hat\sigma^2_{\hat\tau_n},d^2/z_\alpha^2),\\
    \text{FWCID, conservative}: &\quad g(\hat\sigma^2_{\mathcal{P}_n^{ub}}\kappa_n,d^2/z_\alpha^2),\\
    \text{FWCID, always-valid}: &\quad g(v_n^2, d^2/\psi_n^2(\rho,\alpha)).
\end{align*}

We study eight different data-generating processes (DGPs), where treatment is decided by a standard Bernoulli trial with probability of treatment equaling $p=.5$. We have the following eight DGPs:
\begin{align*}
    DGP 1:&\quad  a_1Y(0)\sim Normal(0,1), \quad &&a_1Y(1)\sim Normal(0,1) \\
    DGP 2:&\quad  a_2Y(0)\sim Lognormal(0,1), \quad &&a_2Y(1)\sim Lognormal(0,1) \\
    DGP 3:&\quad  a_3Y(0)\sim Normal(0,1), \quad &&a_3Y(1)\sim Normal(0.2,1) \\
    DGP 4:&\quad  a_4Y(0)\sim Lognormal(0,1), \quad &&a_4Y(1)\sim Lognormal(0,1) + .2 \\
    DGP 5:&\quad  a_5Y(0)\sim Lognormal(0,1), \quad &&a_5Y(1)\sim Lognormal(c_1,1) \\
    DGP 6:&\quad  a_6Y(0)\sim Gamma(1,1),\quad  &&a_6Y(1)\sim Lognormal(c_2,9/16)\\
    DGP 7:&\quad  a_7Y(0)\sim Gamma(1,1), \quad &&a_7Y(1)\sim Gamma(1,c_3) \\
    DGP 8:&\quad  a_8Y(0)\sim Gamma(1,1),\quad  &&a_8Y(1)\sim Gamma(c_4,1).
\end{align*}
The constants $a_j$ are set such that 
\begin{equation}
    \var(Y(0))/(1-p) + \var(Y(1))/p=1, \label{eq:norm_var}
\end{equation}
to facilitate easier comparison of results between simulations of different DGPs. The first two DGPs have no treatment effects, whereas the next two have a homogeneous effect of $\tau=.2$. For the last four DGPs, the treatment effect is heterogeneous where $c_1=1.266\ldots,c_2=-.0949\ldots,c_3=1.223\ldots,c_4=1.210\ldots$ are set such that $\tau=.2$.

We perform simulations for different values of $d$, where $d$ range from $.01$ to $.5$ in increments of $.01$. For each value of $d$ and DGP, we perform 100,000 replications. Let $N^*$ be the sample size for a fixed-sample confidence interval. The asymptotic value of $d$ for a fixed-sample confidence interval is
\begin{equation}
    d^* = \sqrt{\var(Y(0)/((1-p)N^*) + \var(Y(1)/(pN^*)} \cdot z_{\alpha/2}.
\end{equation}
Using the fact that $p=.5$ together with equation \eqref{eq:norm_var}, we can solve for $N^*$ as
\begin{equation}
    N^* = \frac{4 z_{\alpha/2}^2}{{d^*}^2}.
\end{equation}
Let $\overline{N}$ be the value of $N^*$ we get if we replace $d^*$ with $d$. We set $\alpha=.1$ and so for $d=.01,.02,\ldots,.49,.50$, we get $\overline{N}=\num{108222}; \num{27055};\ldots;45;43$, which gives a rough estimate of how big the sample sizes should be for the different simulations. We let $\alpha_c=.1$, whereas $\rho$ is set such that $\phi_n$ and $\psi_n$ are minimized for $n=\overline{N}$, respectively (in line with the suggestion of \citealt{waudby2021time}).

\subsection{Simulation results for FWCID}

We show three sets of results: one for bias (i.e., average treatment effect estimate minus population average treatment effect; Figure \ref{fig:bias}), one for coverage of a 90\% confidence interval (Figure \ref{fig:coverage}) and one for relative sample size, $N/\overline{N}$ (Figure \ref{fig:nratio}). The asymptotic results are for when $d\rightarrow 0$: hence the convergence goes from right to left in each graph. In general, we have enough replications to make sampling variation negligible; hence, no standard errors or intervals are included in the graphs.

\begin{figure}
    \centering
    \includegraphics{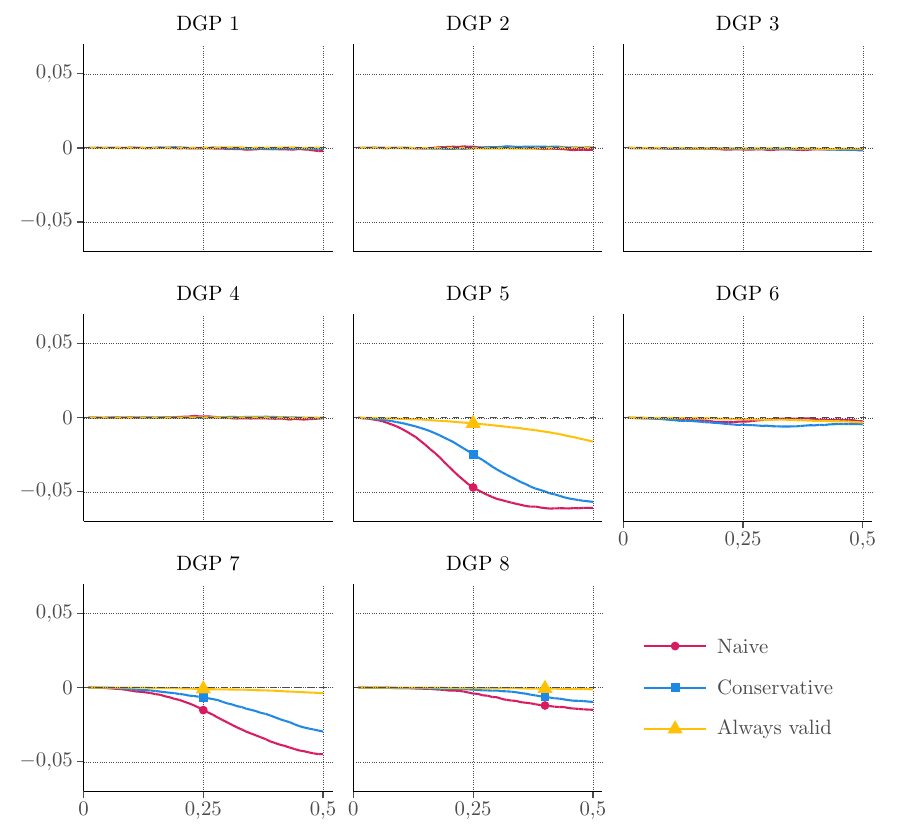}
    \caption{Bias ($y$-axis) of the difference-in-means estimator as a function of $d$ ($x$-axis) for the FWCID}
    \label{fig:bias} 
\end{figure}

Figure \ref{fig:bias} displays the bias of the difference-in-means estimator. In line with Remark \ref{rem:unbiasedness}, we see that when treatment effects are homogeneous, there is no bias at all. With heterogeneous treatment effects, there is bias that depend on the DGP, but, as expected, this bias disappears as $d \rightarrow 0$. Generally, the naive design has the largest bias followed by the conservative design with the always-valid having the smallest bias. This is likely due to the fact that sample sizes for a given $d$ is increasing in that order.

\begin{figure}
    \centering
    \includegraphics{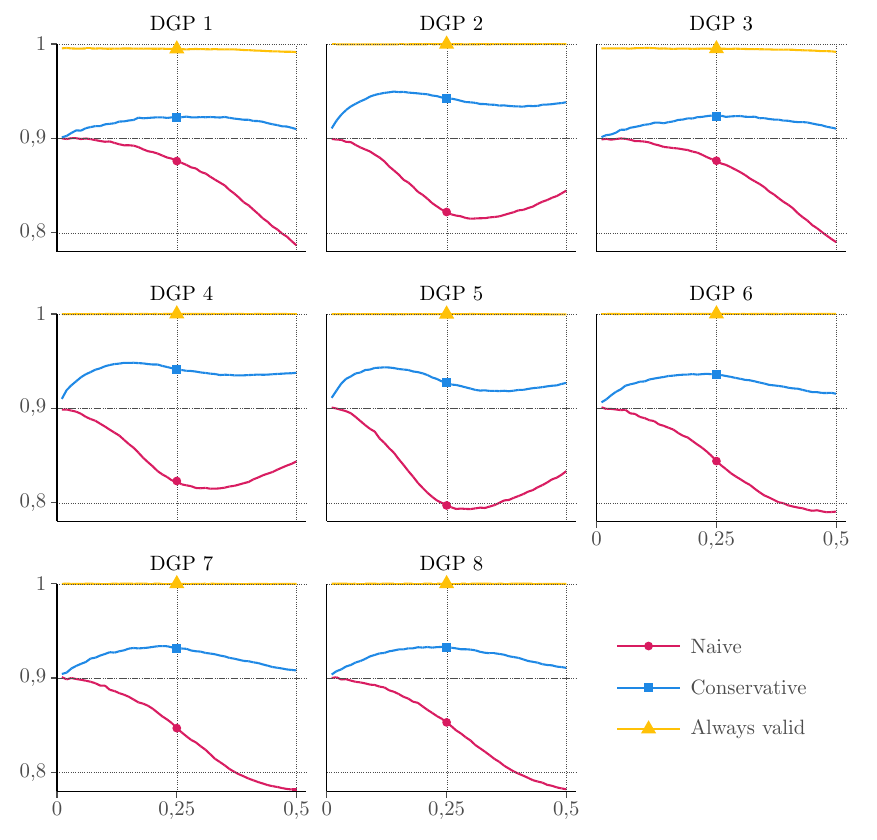}
    \caption{Coverage of a 90\% confidence interval ($y$-axis) of the treatment effect at the stopping point as a function of $d$ ($x$-axis) for the FWCID}
    \label{fig:coverage}
\end{figure}

Turning to coverage of the 90\% confidence interval, we see that the naive interval generally undershoots but converges to 90\%, where the opposite is generally the case for the conservative confidence interval. As expected, the always-valid interval has too wide coverage in all settings with the probability that the interval covers the true parameter being almost one for all DGPs.

\begin{figure}
    \centering
    \includegraphics{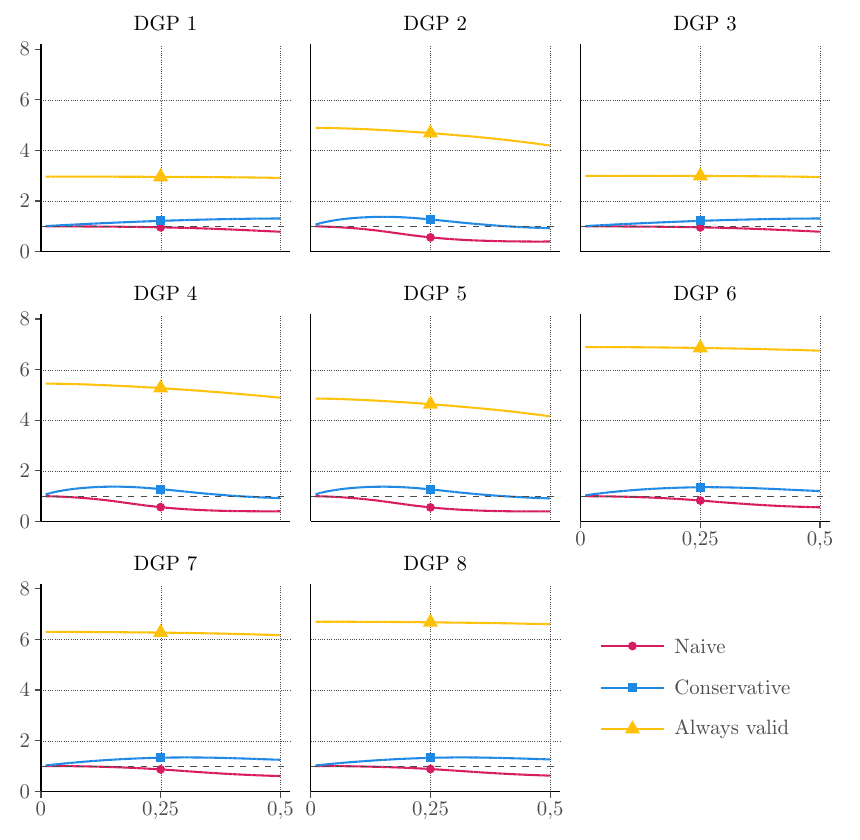}
    \caption{Relative sample size, $E(N)/\overline{N}$ ($y$-axis) as a function of $d$ ($x$-axis) for the FWCID}
    \label{fig:nratio}
\end{figure}

Finally, the explanation for the pattern in Figure \ref{fig:coverage} is given in Figure \ref{fig:nratio}. Here the ratio of the average sample size to the sample size of a fixed-sample design is plotted against $d$. We see that the sample size of the naive design is generally smaller, as expected, but the ratio converges to one as $d \rightarrow 0$ (see Remark \ref{rem:as_eff}), with the fastest convergence happening when data are normal. For the conservative design, convergence generally comes from above as desired and, again, convergence is fastest for normal data. For the always-valid design, the relative sample size is much larger, especially when data follow a skewed distribution. This conservative behavior is a known property of the always-valid CI sequences due to the LIL.


\subsection{Simulation results for hypothesis testing}
To evaluate the FPD, we perform simulations for the six DGPs which have an average treatment effect of .2. We set $\alpha=.05$, $\alpha_c=.1$ and let $\tau_{H_0}=0$, while $\tau_{H_1}$ (and, hence, $\tau_d$) is varied from .1 to .4 in increments of .01. We also need to set a maximum sample size, $N^{max}$, which we set to either $\overline{N}$, $1.5\overline{N}$ or $2\overline{N}$, with $\overline{N}$ being defined as
\begin{equation}
    \overline{N} := \frac{(z_\alpha + z_\beta)^2}{\tau_d^2 p (1-p)}.
\end{equation}
We let the null be $\tau \leq t_{H0}$ and we have $p=.5$. Aside from the two FPD designs (naive and conservative) we include two other designs. The first is the always-valid approach, where a sequence of lower confidence intervals are formed (i.e., the upper bound of the interval is $+\infty$) and the experiment is stopped and the null is rejected as soon as the interval does not cover $\tau_{H_0}$. The second is the commonly used group sequential test (GST, \citealt{lan_discrete_1983}) with alpha spending function $\alpha_n = \alpha \cdot (n/N^{max})^2$, which gives a sequence of critical $z$-values $z_n$ where the experiment is stopped and the null is rejected as soon as the standard $z$-statistic exceeds $z_n$. In summary, the four stopping rules in the simulations are:
\begin{align*}
    \text{FPD, naive}: & \quad g(\hat\sigma^2_{\hat\tau_n},\tau_d^2/(z_\alpha+z_\beta)^2),\\
    \text{FPD, conservative}: &\quad g(\hat\sigma^2_{\mathcal{P}_n^{ub}}\kappa_n,\tau_d^2/(z_\alpha+z_\beta)^2),\\
    \text{Always-valid test}: &\quad g(-(\hat\tau_n - \tau_{H_0})/v_n,-\phi_n(\rho,\alpha)),\\
    \text{GST}: &\quad g(-(\hat\tau_n-\tau_{H_0})/\sigma_{\hat\tau_n}, -z_n).
\end{align*}
The minus signs for the last two approaches come from the fact that the stopping rule as defined in Definition \ref{def:stopping rule} is to stop as soon as the function of data is \emph{smaller} than the sequence of constants. Once again, $\rho$ for the conservative FPD and the always-valid test are set to minimize $\phi_n$, for $n=\overline{N}$ ($\phi_n$ is used instead of $\psi_n$ here for the always-valid approach since we are only considering one-sided tests). Note that the always-valid test is not a FPD because we stop on significance and not on confidence width.\footnote{One could use the always-valid confidence interval to construct a FPD, but, as shown in the simulations for the FWCID, it will have very poor properties with power essentially zero unless $N^{max}$ is set very high.}

We begin by studying bias in Figure \ref{fig:bias_test}. As expected, both the naive and conservative estimators are unbiased with homogeneous treatment effects (DGP 3 and 4) and approximately unbiased once the sample size gets larger (which is the case when $\tau_d$, and hence $\tau_{H_1}$, gets smaller) for the other DGPs. Also as expected, the GST is severely biased, whereas the always-valid test is somewhere in between.

\begin{figure}
    \centering
    \includegraphics{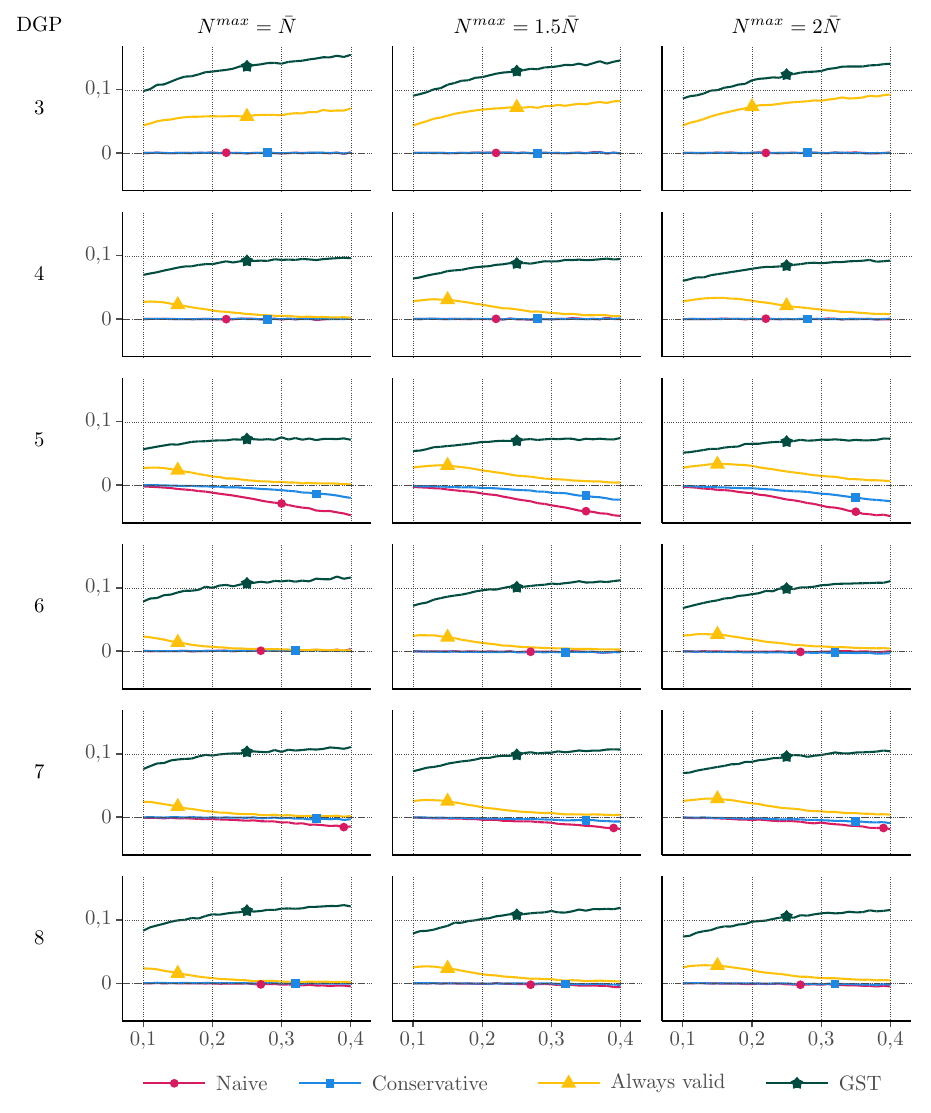}
    \caption{Bias ($y$-axis) for hypothesis test as a function of $\tau_{H_1}$ ($x$-axis)}
    \label{fig:bias_test} 
\end{figure}

Turning to power (Figure \ref{fig:power_test}), we first note that the results depend heavily on which $N^{max}$ that is chosen. For $N^{max}=\overline{N}$, the conservative estimator has low power because it is unlikely that the desired confidence width is achieved at the time of $n=N^{max}$. The always-valid approach performs better with a low $\tau_{H_1}$ (large sample size) because with some probability, the null will be rejected also with a relatively large confidence width because the treatment effect estimator is unexpectedly large (which also means the estimator is biased, see Figure \ref{fig:bias_test}). As expected the GST performs the best.

\begin{figure}
    \centering
    \includegraphics{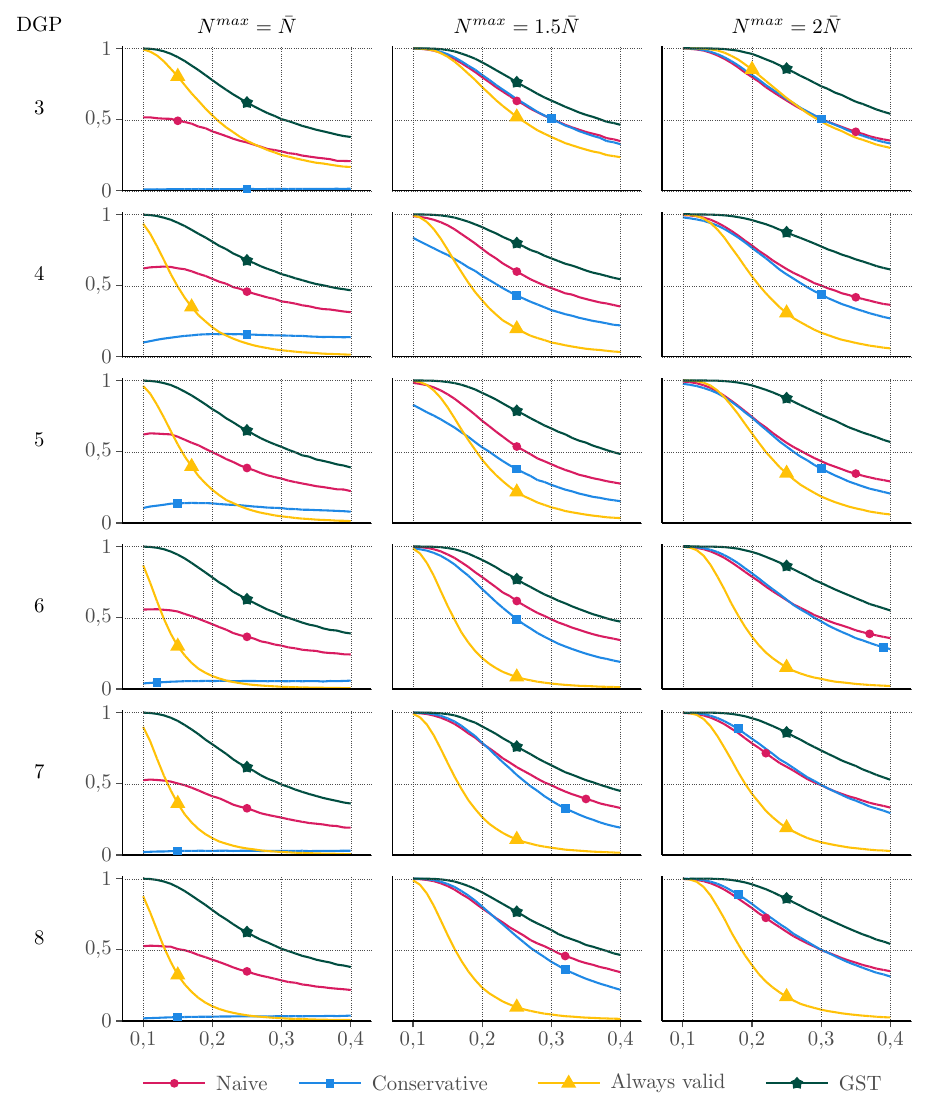}
    \caption{Power ($y$-axis) for hypothesis tests as a function of $\tau_{H_1}$ ($x$-axis)}
    \label{fig:power_test} 
\end{figure}

Note that, according to the theoretical result, the power of both the naive and conservative tests should be $1-\beta=.8$ when $\tau_d=\tau_{H_1}=\tau=.2$ as long as $N<N^{max}$. For $N^{max}=2\overline{N}$, we see that that indeed seems to be the case. The DGP with power furthest from .8 is DGP 5, where the power is around .74 for both the naive and conservative test which occurs because there are some simulations where $N=2\overline{N}$, because the distributions of $Y(0)$ and $Y(1)$ are very skewed.

Turning to the cases where $N^{max} > \overline{N}$, we see that the two FPDs perform better than the always-valid approach. The reason is that with these stopping rules, it is highly likely that the desired confidence width is achieved before $N^{max}$, and so the power will depend almost exclusively on whether the treatment effect estimate is large enough at that point. As expected, the GST continues to have the highest power throughout.

Finally, looking at the average sample size at which sampling is stopped (Figure \ref{fig:nratio_test}), we see that for $N^{max}=\overline{N}$, the conservative stopping rule essentially never binds (i.e., sampling almost never stops before $n=N^{max})$. The same is true for the always-valid approach for large $\tau_{H_1}$. GST generally stops earliest, although it depends on the DGP. For $N^{max}>\overline{N}$, the picture is a bit different. Here the naive stopping rule generally leads to the earliest stopping for large $\tau_{H_1}$ and even the conservative approach can stop relatively early. For small $\tau_{H_1}$, the GST once again stops earliest.

\begin{figure}
    \centering
    \includegraphics{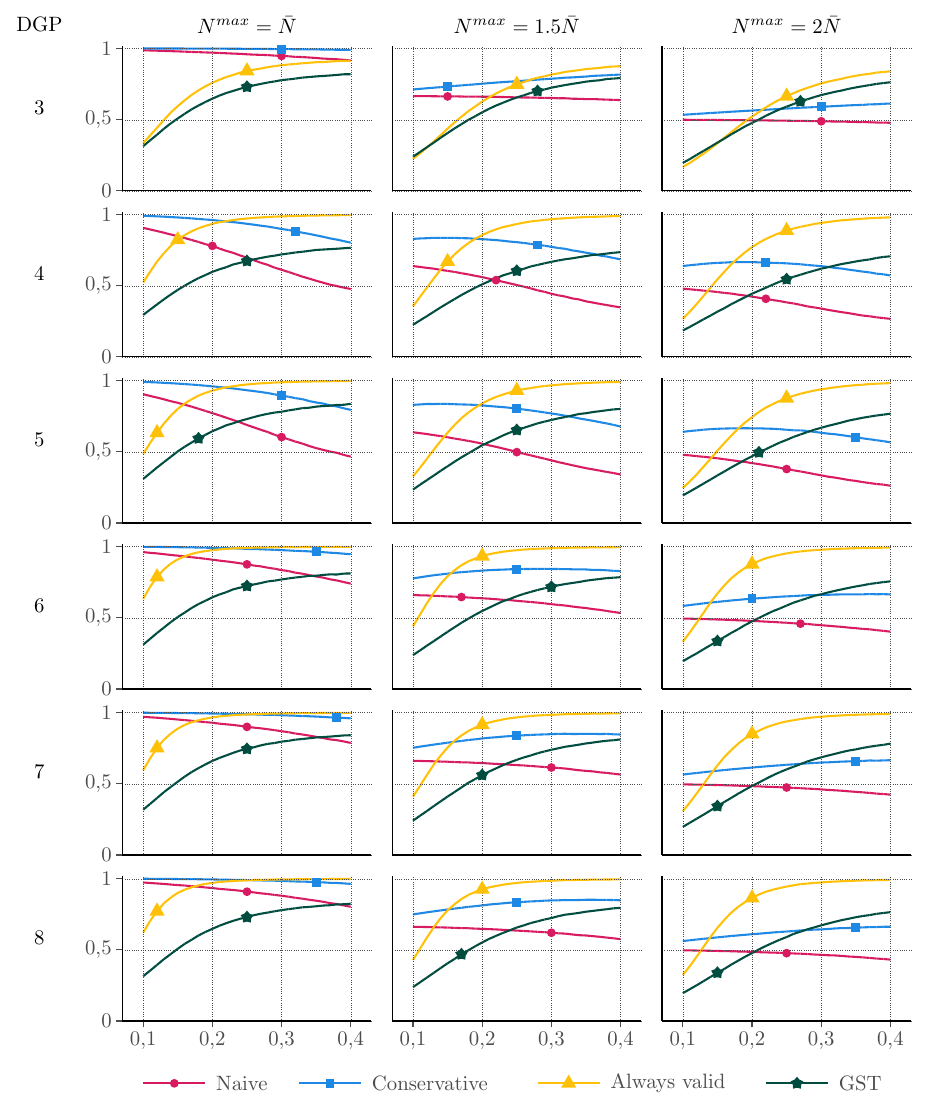}
    \caption{$E(N)/N^{max}$ ($y$-axis) for hypothesis tests as a function of $\tau_{H_1}$ ($x$-axis)}
    \label{fig:nratio_test} 
\end{figure}

\subsection{Conclusions from the simulations}
For forming a confidence interval of a fixed width---the FWCID---we compare three different methods: the naive, conservative and always-valid approaches. All three are asymptotically valid in the sense that coverage is guaranteed to be at least the desired level (and for the first two, it is asymptotically guaranteed to be exactly at that level). The finite-sample simulations indicate that the naive interval will generally have a coverage smaller than the desired level, whereas the conservative and always-valid will have a coverage larger than the desired level. In most cases, the bias is of minor importance, especially for relatively larger sample sizes (smaller confidence widths).

The big difference between the conservative and always-valid approaches comes when comparing the required sample sizes. The always-valid approach requires between two to seven times as much data as the conservative approach. Our conclusion is therefore that the conservative fixed-width confidence interval is the desirable approach since it gives asymptotic guarantees of both coverage and efficiency, while generally being conservative and at the same time not using too much data for small sample sizes. The naive approach will generally work well for large sample sizes, whereas it is hard to motivate the use of the always-valid confidence interval for FWCID.

For hypothesis testing, we compare the FPD for the naive and conservative versions with the group-sequential test and a always-valid test. The FPD has the advantage over the other two in that the point estimator will generally be less biased. However, they will generally be less powerful than the GST, whereas how efficiently they make use of data depends on the specific DGP and parameters used.



\section{Discussion}
In this paper we propose two sequential designs for randomized experiments where the experiment is stopped once a pre-specified precision of the estimate is reached. In the first design---the fixed-width confidence interval design (FWCID)---the experimenter specifies a confidence interval width and the experiment runs until the stopping rule is met, where the stopping rule is set such that the resulting confidence interval gives a coverage equal to the desired level. We propose two versions of this design, a naive and a conservative, where we prove that both versions give consistent estimators of the average treatment effect and that, asymptotically, the resulting interval has correct coverage and makes efficient use of the data.

The second design---the fixed-power design (FPD)---builds on the same idea of stopping once a given precision is reached, but the stopping rule is set such that a given power is reached while maintaining the size of the test. Just as for the FWCID, the treatment effect estimator is consistent and gives a valid confidence interval, and also comes in a naive and a conservative version. Both the FWCID and FPD allow for any randomization scheme (and even some deterministic ones) as long as the share that will be treated asymptotically is fixed. I.e., all standard randomization schemes, such as Bernoulli trials, biased coin designs and randomized block designs are allowed.

We contrast our proposed designs with traditional fixed-sample designs and other sequential designs. In fixed-sample designs (i.e., the sample size is fixed and decided in advance), the key consideration in the design stage is often to decide on a desired power. However, in our experience working with practitioners, power is often a difficult concept to grasp since it can be difficult to come up with a hypothetical treatment effect. Instead, we believe it can be pedagogically easier to focus on the desired precision of the estimator, such as with the FWCID (of course, this pedagogical point can be achieved also in traditional fixed-sample designs, but is perhaps made more explicit in a design explicitly targeting precision).

In addition, for a power analysis in a fixed-sample design to be performed, it is necessary to come up with estimates for the variances of the outcome under treatment and control. At companies like e.g. Spotify, for some experiments it is possible to approximate the variance under control using historical data on the outcome. However, in many settings the historical data is simply not similar enough to the outcome data in the experiment due to growing user bases and changing user behaviors. In addition, it can be difficult to come up with any good estimate for the variance under treatment if the treatment has never been seen historically.

With the FPD, the need to provide these variances are alleviated. However, this comes at a cost: The sample size is no longer fixed which means that there is a risk the experiment runs longer than what was originally planned. On the other hand, this is not necessarily a bug, but rather a feature, of the design. It is not uncommon that experimenters come up with optimistic values in their power calculations to show that it is worth running the experiment, leading to underpowered studies. With the FPD, this possibility is partly restricted (it is still possible to come up with an optimistic hypothetical treatment effect), forcing the experimenter to be more honest. Nevertheless, there is usually a limit to how long an experiment can run, and with the FPD (and FWCID) it is possible to, at any point during the experiment, get an estimate of the required sample size before the experiment concludes. If the decision to terminate an experiment early is based solely on the estimated required sample size, such a decision will not compromise the validity of results for experiments going the distance. In addition, for an experiment that is terminated early, traditional point estimates and confidence intervals will be valid. We stress that this strategy is only valid if early termination is based solely on determining that the estimated required sample size is too large. If the experimenter also base the decision on the treatment effect estimate, all bets are off.

Compared to the most commonly used sequential designs, the advantage of both the FWCID and FPD comes from the fact that the treatment effect estimator is consistent and that valid confidence intervals can be constructed without need for adjustments. The FWCID can be implemented using already established ``always-valid'' approaches for confidence sequences, although the drawback with such approaches is that they are typically less efficient compared to the designs we propose.

In this paper, we have only considered the case with one treatment and one control condition, as well as only studying the difference-in-means estimator. We believe that the main results of this paper could be extended to cover more general settings with more treatment conditions and many different estimators and estimands. We leave the study of such situations for future research.

\subsection*{Acknowledgement}
The authors thank Steve R. Howard for feedback and input to this paper. 
\subsection*{Replication code}
Replication code for the Monte Carlo simulations are available at \url{https://github.com/mattiasnordin/precision-based-designs}

\bibliography{ref.bib}

\end{document}